\documentclass[11pt]{article}

\usepackage{amssymb}
\usepackage{amsfonts}
\usepackage{amsmath}
\usepackage[authoryear,round]{natbib}
\usepackage[dvipsnames]{xcolor}
\usepackage{authblk}
\usepackage[a4paper,textwidth=16cm,textheight=24cm]{geometry}
\usepackage{graphicx}
\usepackage{enumerate}

\newtheorem{theorem}{Theorem}[section]
\newtheorem{proposition}[theorem]{Proposition}
\newtheorem{definition}[theorem]{Definition}
\newtheorem{corollary}[theorem]{Corollary}

\newtheorem{lemma}[theorem]{Lemma}
\newtheorem{remark}[theorem]{Remark}
\newtheorem{example}[theorem]{Example}
\newenvironment{proof}[1][Proof]{\textbf{#1. }}{\hfill \rule{0.55em}{0.5em}\newline}

\newcommand{\stkout}[1]{\ifmmode\text{\sout{\ensuremath{#1}}}\else\sout{#1}\fi}
\newcommand{\Fbar}{\overline{F}}

\newcommand{\Tbar}{\overline{T}}

\newcommand{\nsG}{G_{\beta,\theta}}

\newcommand{\leqst}{\leq_{st}}

\newcommand{\geqst}{\geq_{st}}

\newcommand{\PR}{\mathcal{P}(\mathbb{R})}
\newcommand{\ToMO}{T_o}
\newcommand{\balpha}{\boldsymbol{\alpha}}
\newcommand{\bbeta}{\boldsymbol{\beta}}

\begin{document}

\title{Stochastic Ordering of Dependent Systems under Transformation Models and Archimedean Copulas}
\date{}
\author[1]{Idir Arab\thanks{idir.bhh@gmail.com, ORCID: 0000-0002-9408-684X}}

\author[2]{Milto Hadjikyriakou\thanks{MHadjikyriakou@uclan.ac.uk, ORCID: 0000-0001-5672-7792}}

\author[1]{Paulo Eduardo Oliveira\thanks{Email: paulo@mat.uc.pt, ORCID: 0000-0001-7217-5705}\thanks{P.E.O. acknowledges financial support by the Centre for Mathematics of the University of Coimbra (CMUC, https://doi.org/10.54499/UID/00324/2025)
under the Portuguese Foundation for Science and Technology (FCT),
Grants UID/00324/2025 and UID/PRR/00324/2025.}
}

\affil[1]{CMUC, Department of Mathematics, University of Coimbra, Portugal}

\affil[2]{School of Sciences, University of Central Lancashire, Cyprus}

\maketitle

\begin{abstract}
We study stochastic ordering of system lifetimes with dependent and heterogeneous components whose marginal distributions are obtained through transformations of a common baseline. The dependence structure is modeled via Archimedean copulas, allowing for a unified treatment of several transformation-based models, including proportional hazard, proportional reversed hazard rate and proportional odds families. For parallel, series and $(n-k)$-out-of-$n$ systems, we derive conditions for stochastic dominance based on monotonicity of the transformation and structural properties of the copula generators, formulated through super-additivity and Schur-type arguments.

The results provide tractable criteria that extend existing comparisons beyond independence and illustrate the combined effect of dependence and parameter heterogeneity on system reliability.
\end{abstract}


\smallskip


\section{Introduction}

Comparing the behaviour of lifetime distributions between complex systems has attracted considerable interest due to its applications in reliability and the overall performance of devices constructed from multiple components. It is clear that the lifetime of each component reflects on the collective behaviour of the system. Moreover, the way components interact with each other—either due to the system’s physical properties or statistical dependence between components—is crucial for characterising the system itself. Since lifetimes are random, any comparison must capture the stochastic nature of the problem and be based on an appropriate version of a stochastic order between the random variables representing the lifetimes of interest.

The basic notions and properties of stochastic orders may be found in the monograph by \cite{shaked2007}. Depending on the structure of a particular complex system, its lifetime may be described using order statistics derived from the component lifetimes. Consequently, one is often interested in the stochastic comparison of order statistics. A review of relevant results can be found in \cite{kochar2012} and \cite{KocharXu2013}, with the latter studying the effect of dependence between components on the overall lifetime and how to exert some control over that dependence.

The physical properties of a system are typically reflected in how the lifetimes of its individual components determine the overall system lifetime. Among the most widely studied configurations are parallel and series systems, where the system operates if at least one component is functioning (parallel) or only if all components are functioning (series). Intermediate configurations lead to the class of $(n-k)$-out-of-$n$ systems, in which the system continues to operate as long as at least $n-k$ of the $n$ components remain alive. There is an extensive literature on $(n-k)$-out-of-$n$ systems, which is closely tied to the properties of $(n-k)$-th order statistics. Most of this work assumes independent component lifetimes, due to the mathematical challenges associated with deriving order statistic distributions in the presence of dependence.

Focusing on parallel or series systems, it is not surprising that the first results in the literature compared such systems with respect to several stochastic orderings, assuming independent components with specific lifetime distributions, as in \cite{DykstraKochar1997} or \cite{KhalediKochar2000} for exponentially distributed components. More recently, \cite{ZhaoZhangQiao2016} and \cite{NaqviDingZhao2022} considered other distribution families for the component lifetimes. A different type of result, characterising non-comparability for a specific criterion, may be found in \cite{NonComp2020}, which provides results for the proportional hazard rates (PHR) and proportional reversed hazard rates (PRHR) models that we will consider later.

The stochastic comparison of systems with dependent, possibly heterogeneous, components has been addressed using characterisations of joint distributions described through specific copula functions, in particular Archimedean copulas. In this direction, \cite{LF2015} proved stochastic comparison results when component lifetimes follow a PHR model; \cite{PanjaKunduPradhan2021} considered stochastic domination with components lifetime distributions from a proportional odds model; and the PRHR model has been studied in \cite{Datta_Gupta2024}, among other models for the individual distributions, with emphasis in comparing systems with respect to dispersive or hazard rate orders. Further related results for distorted distributions and coherent systems may be found in \cite{navarro2017}, while \cite{AmiriBalaJama2022} and \cite{A-S-Kh_Iza2024} provide results for more specific families of component lifetimes. 

In this paper, we develop a general framework for stochastic comparison of system lifetimes with dependent and heterogeneous components, whose marginal distributions arise from transformations of a common baseline and whose dependence structure is described by Archimedean copulas. This setting includes several widely used models, such as the PHR and PRHR families, within a unified formulation. For parallel and series systems, we establish stochastic ordering results based on monotonicity properties of the transformation and structural conditions on the copula generators. We further consider $(n-k)$-out-of-$n$ systems, for which a different approach is required, deriving explicit representations of the system lifetime distribution under Archimedean dependence and obtaining stochastic ordering results when the joint distribution is described through the same copula. The applicability of the proposed framework is illustrated through several transformation models, highlighting the combined effect of dependence and parameter heterogeneity on system reliability.

\section{Notations and some preliminary results}
We now introduce some general concepts and terminology to be used in the sequel. The set of all cumulative distribution functions in $\mathbb{R}$ is denoted as $\PR$. Given a random variable $X$ we represent with $F_X(\in\PR)$ and $\Fbar_X$ its cumulative distribution function and survival function, respectively, and may denote this as $X\sim F_X$. 
Analogously, given a $n$-dimensional random vector $\mathbf{X}=(X_1,\ldots,X_n)$ the distribution function is represented by $F_{(X_1,\ldots,X_n)}$, while the marginals are denoted as $F_{X_i}$, $i=1,\ldots,n$. Further, we shall represent as $\Fbar_{(X_1,\ldots,X_n)}(x_1,\ldots,x_n)={\rm P}(X_1>x_1,\ldots,X_n>x_n)$. Recall that a real valued function $g$ is super-additive if, for every $x,\,y\geq 0$, $g(x+y)\geq g(x)+g(y)$. It is well known that if $g(0)\leq 0$, then the convexity of $g$ implies its super-additivity.

We recall the definition of stochastic dominance, the key comparison notion between random variables that will be studied.

\begin{definition}
\label{def:st}
Consider two random variables $X_1\sim F_{X_1}$ and $X_2\sim F_{X_2}$. We say that $X_1$ is smaller than $X_2$ in the usual stochastic order, denoted by $X_1\leqst X_2$, if $\Fbar_{X_1}(x) \leq \Fbar_{X_2}(x)$ for every $x \in \mathbb{R}$. 
\end{definition}

\subsection{Joint distributions and copula representations}
\label{subsec:copulas}
{To study systems with dependent components, we introduce a copula-based representation of joint distribution functions}. It is well known, from the Sklar's Theorem (see \cite{Sklar59}), that, given a random vector $\mathbf{X}=(X_1,\ldots,X_n)$, there exist $n$-dimensional distribution functions, $C$ and $\widehat{C}$, with uniform marginals such that
\begin{equation}
\label{eq:cop-rep}
\begin{array}{l}
\displaystyle F_{(X_1,\ldots,X_n)}\left(x_1, \ldots, x_n\right)=C\left(F_{X_1}(x_1), \ldots, F_{X_n}(x_n)\right), \\[2ex]
\displaystyle \Fbar_{(X_1,\ldots,X_n)}\left(x_1, \ldots, x_n\right)=\widehat{C}\left(\Fbar_{X_1}(x_1), \ldots, \Fbar_{X_n}(x_n)\right).
\end{array}
\end{equation}
The functions $C$ and $\widehat{C}$ are the \textit{copula} and \textit{survival copula}, respectively, of the vector $\mathbf{X}$. A particular class of copulas, introduced next, is of special interest. We need to recall a preparatory notion, before the actual definition of the class of interest.
\begin{definition}
A real function $f$ is called $d$-monotone in $(a, b)$, where $a, b \in \mathbb{R}$ and $d \geq 2$, if it is differentiable there up to the order $d-2$ and the derivatives satisfy $(-1)^{i} f^{(i)}(x) \geq 0$, $i=0,1, \ldots, d-2$ and $(-1)^{d-2} f^{(d-2)}(x)$ is nonincreasing and convex in $(a, b)$.
\end{definition}

\begin{definition}
\label{def:arch}
Suppose $\psi:[0, \infty) \rightarrow[0,1]$ is a continuous $d$-monotone function with $\psi(0)=1$ and $\lim_{x\to \infty}\psi(x)=0$. 
Then, a copula $C_{\psi}$ is called an Archimedean copula with generator $\psi$ if it can be written as $C_{\psi}\left(v_{1}, \ldots, v_{n}\right)=\psi\left(\phi\left(v_{1}\right)+\cdots+\phi\left(v_{n}\right)\right)$ for all $v_{i} \in[0,1]$, $i =1, \ldots, n$, 
where $\phi=$ $\psi^{-1}=\sup \{x \in \mathbb{R}: \psi(x)>v\}$ is the right continuous inverse of the generator $\psi$.
\end{definition}
\noindent In the sequel, we focus on several well–known Archimedean copula families that are widely used in dependence modelling and reliability analysis. For convenience, we summarise their generators, inverse generators and parameter ranges in Table~\ref{tab:archimedean}. 
\begin{table}[h!]
	\centering
	\caption{Generators, inverse generators and parameter ranges of selected Archimedean copulas.}
	\label{tab:archimedean}
	\medskip
	\begin{tabular}{p{3cm}lll}
		\hline
		Copula & Generator $\psi(t)$ & Inverse $\phi(u)=\psi^{-1}(u)$ & Parameter range \\
		\hline
		Independence 
		& $e^{-t}$ 
		& $-\log u$ 
		& none \\[6pt]
		
		Clayton 
		& $(1+\gamma t)^{-1/\gamma}$ 
		& $\dfrac{1}{\gamma}(u^{-\gamma}-1)$ 
		& $\gamma \ge -1,\ \gamma \neq 0$ \\[10pt]

		Frank
        & $-\frac1\gamma\log\left(e^{-t}(e^{-\gamma}-1)+1\right)$
        & $-\log\left(\tfrac{e^{-\gamma u}-1}{e^{-\gamma}-1}\right)$
		& $\gamma\not=0$
        \\[10pt]
		
		Gumbel 
		& $\exp(-t^{1/\gamma})$ 
		& $(-\log u)^{\gamma}$ 
		& $\gamma \ge 1$ \\[10pt]
		
		Ali--Mikhail--Haq (AMH) 
		& $\dfrac{1-\gamma}{e^{t}-\gamma}$ 
		& $\log\left(\dfrac{1-\gamma(1-u)}{u}\right)$ 
		& $\gamma \in (-1,1)$ \\
		\hline
	\end{tabular}
\end{table}

The following ordering result {for copulas is} of interest.
\begin{lemma}[\cite{LF2015}, Lemma~A.1]
\label{cop}
For two $n$-dimensional Archimedean copulas $C_{\psi_1}$ and $C_{\psi_2}$, if $\phi_2 \circ \psi_1$ is super-additive, then
$C_{\psi_1}(\boldsymbol{u}) \leq C_{\psi_2}(\boldsymbol{u})$, for $\boldsymbol{u} \in [0,1]^n$.
\end{lemma}
We show a few examples illustrating the scope of applicability of Lemma~\ref{cop}.
\begin{example}
Assume that $\psi_1(t) = e^{-t}$, the generator of the independence copula. {For generators $\psi_2$ within the above-mentioned families, the super-additivity assumption holds for a wide range of parameter values}.
    \begin{enumerate}
    \item
    The Clayton case with $\gamma_2\geq0$.
    We have that $\phi_2\circ\psi_1(t) = \frac{1}{\gamma_2} \left(e^{\gamma_2 t} - 1\right)$, which is easily seen to be convex. As $\phi_2\circ\psi_1(0)=0$, the super-additivity follows. 
    
    \item
    The Frank case with $\gamma_2>0$. For this copula we find
    $$
    \phi_2\circ\psi_1(t)=-\log\left(\frac{e^{-\gamma_2 e^{-t}}-1}{e^{-\gamma_2}-1}\right),
    $$
    that is easily seen to be convex for $t\geq 0$, hence, as $\phi_2\circ\psi_1(0)=0$, the super-additivity follows.
    
    \item
    The Gumbel case with $\gamma_2\geq1$. Then
	$\phi_2\circ\psi_1(t)=(-\log e^{-t})^{\gamma_2}=t^{\gamma_2}$,
    which is super-additive on $[0,\infty)$.
    
    \item
    The AMH case with $\gamma_2\in(0,1)$.
    Then
	$$
	\phi_2\circ\psi_1(t)
	=\log\left(\frac{1-\gamma_2(1-e^{-t})}{e^{-t}}\right)
	=\log\left((1-\gamma_2)e^{t}+\gamma_2\right).
	$$
	Again, we find a convex function such that $\phi_2\circ\psi_1(0)=0$, thus the super-additivity follows.
    \end{enumerate}
\end{example}
	
\begin{example}
\label{ex:fams}
Consider now two generators $\psi_1$ and $\psi_2$ from the same family of copulas, with parameters $\gamma_1$ and $\gamma_2$. With the appropriate range of variability for the parameters, the composition $\phi_2\circ\psi_1$ is still super-additive.
    \begin{enumerate}
    \item
    The Clayton family with $\gamma_2\geq\gamma_1>0$. In this case, we have $\phi_2\circ\psi_1(t) = \frac{1}{\gamma_2}\left((1+\gamma_1 t)^{\gamma_2/\gamma_1} - 1\right)$. For the given choice of parameters, $\phi_2\circ\psi_1$ is convex so, as $\phi_2\circ\psi_1(0)=0$, this function is super-additive.
    
    \item
    The Frank copulas with $\gamma_2\geq\gamma_1$. A direct calculation gives
    $$
    \phi_2\circ\psi_1(t)=-\log\left(\frac{\left(e^{-t}(e^{-\gamma_1}-1)+1\right)^{\gamma_2/\gamma_1}-1}{e^{-\gamma_2}-1}\right),
    $$
     which is easily seen to be convex when $\gamma_2\geq\gamma_1$, hence, again as $\phi_2\circ\psi_1(0)=0$, it follows that $\phi_2\circ\psi_1$ is super-additive.
    
    \item
    The Gumbel family with $\gamma_2\geq\gamma_1\geq 1$. We find that $\phi_2 \circ \psi_1(t)= t^{\gamma_2/\gamma_1}$ which, being convex and null at the origin, is super-additive.
    
    \item
    The AMH copulas with $1\geq\gamma_2\geq\gamma_1\geq0$. Some simple algebra shows that
    $$
    \phi_2\circ\psi_1(t)
	=\log\left(\frac{1-\gamma_2}{1-\gamma_1}e^{t}+\frac{\gamma_2-\gamma_1}{1-\gamma_1}\right)
	=\log\big(ae^{t}+b\big),
    $$
    where $a=\frac{1-\gamma_2}{1-\gamma_1}$ and $b=\frac{\gamma_2-\gamma_1}{1-\gamma_1}$, so $a+b=1$.
	As $\gamma_2\geq\gamma_1$, it follows that $a\in(0,1]$ and $b\in[0,1)$. Therefore, for $x,y\ge 0$ we have
	$$
	\log(ae^{x+y}+b)\ge \log(ae^{x}+b)+\log(ae^{y}+b)
	\quad\Leftrightarrow\quad 
	ae^{x+y}+b\ge (ae^{x}+b)(ae^{y}+b).
	$$
	Expanding and rearranging the terms,
	$$
	ae^{x+y}+b-(ae^{x}+b)(ae^{y}+b)
	=a(1-a)\big(e^{x+y}-e^{x}-e^{y}+1\big)
	=a(1-a)(e^{x}-1)(e^{y}-1)\ge 0,
	$$
	which proves super-additivity in this case as well.
    \end{enumerate}
\end{example}

\subsection{Schur-convexity}
Finally, we recall the main tool to derive the stochastic dominance between different distributions {depending on a} family of parameters.

\begin{definition}
\label{def:Schur}
Let $\boldsymbol{\alpha}=(\alpha_1,\ldots,\alpha_n),\,\boldsymbol{\beta}=(\beta_1,\ldots,\beta_n)\in\mathbb{R}^n$.
    \begin{enumerate}
    \item (\cite{marshall2011}, Definition~1.A.1)
     The vector $\boldsymbol{\alpha}$ is said to be majorized by another vector $\boldsymbol{\beta}$, denoted by $\boldsymbol{\alpha} \overset{m}{\preceq}\boldsymbol{\beta}$, if for each $\ell=1, \ldots, n-1$, we have $\sum_{i=1}^{\ell} \alpha_{i:n} \geq \sum_{i=1}^{\ell} \beta_{i:n}$ and $\sum_{i=1}^{n} \alpha_{i:n}=\sum_{i=1}^{n} \beta_{i:n}$, where $\alpha_{i:n}$ and $\beta_{i:n}$ represent the coordinates of $\boldsymbol{\alpha}$ and $\boldsymbol{\beta}$, respectively, ordered increasingly.
    \item (\cite{marshall2011}, Definition~3.A.1)
	A function $h:A\subseteq \mathbb{R}^{n} \rightarrow \mathbb{R}$ is said to be Schur-convex (Schur-concave) on $A$ if
	$$
	\boldsymbol{\alpha} \overset{m}\succeq \boldsymbol{\beta} \Rightarrow h(\boldsymbol{\alpha}) \geq(\leq)\; h(\boldsymbol{\beta}),
    \quad \text { for all } \boldsymbol{\alpha},\, \boldsymbol{\beta} \in A.
	$$
\end{enumerate}
\end{definition}
The result below provides a connection between majorization and Schur-convexity.
\begin{theorem}[\cite{marshall2011}, Theorem~3.A.4]
	\label{MO1}Suppose $J \subset \mathbb{R}$ is an open interval and $h: J^{n} \rightarrow \mathbb{R}$ is continuously differentiable. Then, $h$ is Schur-convex (Schur-concave) on $J^{n}$ if and only if
	\begin{enumerate}
		\item $h$ is symmetric on $J^{n}$, and
		\item for all $i \neq j$ and all $\boldsymbol{\alpha} \in J^{n}$,
		\begin{equation}
		\label{eq:Schur}
        \left(\alpha_{i}-\alpha_{j}\right)\left(\frac{\partial h(\boldsymbol{\alpha})}{\partial \alpha_{i}}-\frac{\partial h(\boldsymbol{\alpha})}{\partial \alpha_{j}}\right) \geq (\leq)\,0,
		\end{equation}
		where $\frac{\partial h(\boldsymbol{\alpha})}{\partial \alpha_{i}}$ represents the partial derivative of $h$ with respect to its $i$-th argument.
	\end{enumerate}
\end{theorem}
For a thorough discussion on the theory of majorization and Schur functions, one may refer to {\cite{marshall2011}}.

\section{Stochastic dominance for complex systems}
\label{sec:general}
{We shall prove some general stochastic dominance results comparing systems built on different collections of components, whose lifetimes are given by a semiparametric model. Let us start by defining the framework. Let $T:A\times\PR\longrightarrow\PR$, where $A\subset\mathbb{R}$, be some functional and consider collections of random variables $X_i$ and $Y_i$, for $i=1,2,\ldots,n$, with distributions $F_{X_i}=T(\alpha_i,F)$ and $F_{Y_i}=T(\beta_i,F)$, where $\alpha_i,\beta_i\in A$ such that $\mathbf{X}=(X_1,\ldots,X_n)$ and $\mathbf{Y}=(Y_1,\ldots,Y_n)$ have joint distributions described by Archimedean copulas with generators $\psi_1$ and $\psi_2$, respectively.}

{\subsection{Stochastic dominance between parallel systems}}
{The lifetime of the parallel system with components $X_i$, for $i=1,\ldots,n$, is given by $X_{n:n}(\boldsymbol{\alpha})=\max\{X_1,\ldots,X_n\}$, where $\boldsymbol{\alpha}=(\alpha_1,\ldots,\alpha_n)$, whose distribution function is of the form
\begin{equation}
\label{eq:maxX}
h(\boldsymbol{\alpha};\psi_1)(x)=F^{(\balpha)}_{n:n}(x)=\psi_1\left(\sum_{i=1}^n\phi_1(T(\alpha_i,F(x)))\right),
\end{equation}
where $\phi_1$ is the inverse of $\psi_1$. Analogous expressions characterise the parallel system with components $Y_i$, for $i=1,\ldots,n$, with lifetime $Y_{n:n}(\boldsymbol{\beta})$, that has distribution function $h(\boldsymbol{\beta};\psi_2)$, where $\boldsymbol{\beta}=(\beta_1,\ldots,\beta_n)$.}

We prove a first stochastic domination result for parallel systems assuming a strict order relation between the parameters characterising each distribution.
\begin{theorem}
\label{thm:main0}
{Assume the parameter vectors $\boldsymbol{\alpha}$ and $\boldsymbol{\beta}$ are such that $\alpha_i\leq\beta_i$, for $i=1,\ldots,n$.} 
    \begin{enumerate}
    \item
    {If} $T$ is increasing in its first argument that is, $\alpha \leq \beta$ implies that, for every $x\in\mathbb{R}$, $T(\alpha, F) (x) \leq T(\beta, F) (x)$ and the generators are such that $\phi_2\circ\psi_1$ is super-additive{, then} {$X_{n:n}(\boldsymbol{\alpha})\geqst Y_{n:n}(\boldsymbol{\beta})$}.
    \item
    {If} $T$ is decreasing in its first argument that is, $\alpha\leq \beta$ implies that, for every $x\in\mathbb{R}$, $T(\alpha, F) (x) \geq T(\beta, F) (x)$ and the generators are such that $\phi_1\circ\psi_2$ is super-additive{, then} {$X_{n:n}(\boldsymbol{\alpha})\leqst Y_{n:n}(\boldsymbol{\beta})$}.
    \end{enumerate}
\end{theorem}
\begin{proof}
Regarding assertion 1., according to Lemma~\ref{cop}, as $\phi_2\circ\psi_1$ is super-additive, it follows that
\begin{eqnarray*}
F_{X_{n:n}}^{(\balpha)} & = & C_{\psi_1}(T(\alpha_1,F),\ldots,T(\alpha_n,F)) \leq C_{\psi_2}(T(\alpha_1,F),\ldots,T(\alpha_n,F)) \\[1.5ex] 
 & \leq &  C_{\psi_2}(T(\beta_1,F),\ldots,T(\beta_n,F))=F_{Y_{n:n}}^{(\bbeta)},
\end{eqnarray*}
using the increasingness of $T(\cdot,F)$ and the fact that $C_{\psi_2}$, being a copula, is a multivariate distribution function, hence increasing in each coordinate. The second assertion is proved analogously taking into account the reversal of the direction of the super-additivity assumption.
\end{proof}
\begin{remark}
Note that the monotonicity assumptions in Theorem~\ref{thm:main0} may be restated in terms of the usual stochastic order. Indeed, for example, as what regards assertion 1., one may rewrite it as $\alpha\leq\beta$ implies that $T(\alpha,F)\geqst T(\beta,F)$. Of course, assertion 2.\ reverses the direction of this stochastic dominance.
\end{remark}

The requirement 
of coordinatewise increasing parameter vectors 
may be relaxed, using the majorization $\overset{m}{\preceq}$ relation introduced in Definition~\ref{def:Schur}, as proved next.
\begin{theorem}
\label{thm:main1}
{Assume the functional $T$ is differentiable with respect to its first argument, the generators are such that $\phi_2\circ\psi_1$ is super-additive, $\psi_2$ is $d$-monotone with $d\geq 3$, $\boldsymbol{\alpha} \overset{m}\succeq \boldsymbol{\beta}$} and 
\begin{equation}
\label{eq:V-inc}
V(\beta) =\frac{1}{\psi_2^\prime\circ\phi_2(T(\beta,F))}\frac{\partial T}{\partial\beta}(\beta,F),\quad\beta\in A,
\end{equation}
is increasing, then {$X_{n:n}(\boldsymbol{\alpha})\geqst Y_{n:n}(\boldsymbol{\beta})$}.
\end{theorem}
\begin{proof}
According to Lemma~\ref{cop}, as $\phi_2\circ\psi_1$ is super-additive, it follows that {$h(\boldsymbol{\alpha};\psi_1) \leq h(\boldsymbol{\alpha};\psi_2)$}. 
As we are assuming that $\boldsymbol{\alpha} \overset{m}\succeq \boldsymbol{\beta}$, taking into account Definition~\ref{def:Schur}, the conclusion follows if we prove the Schur-concavity of {$h(\boldsymbol{\beta};\psi_2)$} with respect to {$\boldsymbol{\beta}$}, for what we shall verify that the assumptions of Theorem~\ref{MO1} are satisfied. As {$h(\boldsymbol{\beta};\psi_2)$} is obviously symmetric with respect to $\boldsymbol{\beta}$, it remains to show that (\ref{eq:Schur}) holds. Differentiating, we have
$$
\frac{\partial h}{\partial\beta_i}=\psi_2^\prime\left(\sum_{i=1}^n\phi_2(T(\beta_i,F))\right)V(\beta_i),
$$
hence
$$
(\beta_i-\beta_j)\left(\frac{\partial h}{\partial\beta_{i}}-\frac{\partial h}{\partial\beta_{j}}\right) =
\psi_2^\prime\left(\sum_{i=1}^n\phi_2(T(\beta_i,F))\right)(\beta_i-\beta_j)\bigl(V(\beta_i)-V(\beta_j)\bigr)\leq 0,
$$
as $V$ is increasing and $\psi_2$, being the generator of an Archimedean copula, is decreasing.
\end{proof}

For many Archimedean copula generators, namely, for the families of copulas mentioned above, a convenient representation of $\psi^\prime$ is available, allowing for a simplification of the function $V$ defined in (\ref{eq:V-inc}).
\begin{corollary}
\label{cor:main1}
Assume the same conditions as in Theorem~\ref{thm:main1}. 
With $K_2(x)=-\frac{\psi_2(x)}{\psi_2^\prime(x)}$, the same conclusion holds if
(\ref{eq:V-inc}) is replaced by
\begin{equation}
\label{eq:K2-inc}
V^\ast(\beta) =\frac{-K_2\circ\phi_2(T(\beta,F))}{T(\beta,F)}\frac{\partial T}{\partial\beta}(\beta,F),\quad\beta\in A,
\end{equation}
\end{corollary}
\begin{proof}
This is a straightforward consequence of $\phi_2$ being the inverse of $\psi_2$.
\end{proof}

\begin{example}
\label{ex:exp}
As an application of the previous results, choose  as baseline the standard exponential distribution $\mathcal{E}(x)=1-e^{-x}$, for $x\geq 0$, and, for $\beta>0$, $T(\beta,\mathcal{E})(x)=1-(1-\mathcal{E}(x))^\beta=1-e^{-\beta x}$, the exponential distribution with hazard rate $\beta$. Then, with the notation of Theorems~\ref{thm:main0} and \ref{thm:main1}, $\mathbf{X}=(X_1,\ldots,X_n)$ and $\mathbf{Y}=(Y_1,\ldots,Y_n)$ are random vectors with dependent components whose marginal distributions are heterogeneous exponentials. Assume the distributions of $\mathbf{X}$ and $\mathbf{Y}$ are described using a Clayton copula with parameters $\gamma_1$ and $\gamma_2$, respectively, such that $\gamma_2\geq\gamma_1>0$. Then, as referred in Example~\ref{ex:fams}, $\phi_2\circ\psi_1$ is super-additive. Hence, from Theorem~\ref{thm:main0}, if $\alpha_i\leq\beta_i$, $i=1,\ldots,n$, it follows that {$X_{n:n}(\boldsymbol{\alpha})\geqst Y_{n:n}(\boldsymbol{\beta})$}. In case we do not have increasing coordinates for the parameter vectors, but they still satisfy $\boldsymbol{\alpha}\overset{m}\succeq\boldsymbol{\beta}$, we get a conclusion using Theorem~\ref{thm:main1}. Moreover, for the Clayton generator $\psi_2^\prime\circ\phi_2(T(\beta,\mathcal{E}))(x)=-(1-e^{-x})^{1+\gamma_2}$, so
$$
V(\beta)=-x(1-e^{-\beta x})^{-(1+\gamma_2)}e^{-\beta x},
$$
which is increasing with respect to $\beta>0$, so again we conclude that {$X_{n:n}(\boldsymbol{\alpha})\geqst Y_{n:n}(\boldsymbol{\beta})$}.
\end{example}

The characterisation established in Theorem~\ref{thm:main1} provides a tool for identifying extremal configurations of parallel systems under stochastic ordering of their lifetimes. Under rectangular constraints on the parameter space, it shows that the extremal behaviour is achieved by specific boundary configurations, thus yielding an explicit and tractable description of optimal parameter allocations.
\begin{corollary}
\label{cor:boxconst}
Assume the functional $T$ and the generator of the Archimedean copula describing the distribution of $\mathbf{X}$ are such that the function $V$ defined in (\ref{eq:V-inc}) (with $\psi_2=\psi_1$) is increasing. Fix numbers $a < b$ and $c \in [na, nb)$, and define
$$
A_{a,b,c} = \left\{\balpha = (\alpha_1,\dots,\alpha_n) \in [a,b]^n : \sum_{i=1}^n \alpha_i = c \right\},
$$
$$
\overline{\balpha} = \left(\frac{c}{n}, \ldots, \frac{c}{n}\right),\quad q = \left\lfloor \frac{c - na}{b - a} \right\rfloor, 
\quad\mbox{and} \quad
\eta = c - qb - (n-q-1)a.
$$
Then $\eta \in [a,b)$ and
$$
\balpha^* = (\!\!\!\!\underbrace{a,\dots,a}_{n-q-1\text{ times}}\!\!\!\!, \eta, \underbrace{b,\dots,b}_{q\text{ times}}) \in A_{a,b,c}.
$$
Moreover, for every $\balpha \in A_{a,b,c}$, $X_{n:n}(\balpha^*) \ge_{st} X_{n:n}(\balpha) \ge_{st} X_{n:n}(\overline{\balpha})$.
\end{corollary}
\begin{proof}
First, note that $\overline{\balpha} \in A_{a,b,c}$ implies $a \leq \frac{c}{n} < b$, as $c \in [na,nb)$. Moreover, we may write $c - na = q(b-a) + r$, where $0 \leq r < b-a$,	which leads to $\eta = c - qb - (n-q-1)a = a + r$, proving that $\eta \in [a,b)$. Therefore $\balpha^* \in A_{a,b,c}$.
Now, taking into account Theorem~\ref{thm:main1}, it is enough to prove that $\balpha^*\overset{m}\succeq\balpha\overset{m}\succeq\overline{\balpha}$. The latter majorization is obvious, as among vectors with coordinates with a fixed sum, the constant vector is well known to be majorised by any other vector. Hence, we need to verify that, for every $\ell=1,\ldots,n-1$, $\sum_{i=1}^\ell \alpha_{i:n}\geq\sum_{i=1}^\ell \alpha^*_{i:n}$, where $\alpha_{i:n}$ and $\alpha_{i:n}^*$ represent the coordinates of $\balpha$ and $\balpha^*$, respectively, ordered increasingly.
    \begin{enumerate}[(i)]
    \item
    If $1 \leq \ell \leq n-q-1$, then the $\ell$ smallest coordinates of $\balpha^*$ are all equal to $a$, so $\sum_{i=1}^\ell \alpha^*_{i:n} = \ell a$.
	Since every coordinate of $\balpha$ belongs to $[a,b]$, we also have $\alpha_{i:n} \geq a$, for every $i = 1,\ldots,\ell$,
	and therefore $\sum_{i=1}^\ell \alpha_{i:n} \geq \ell a = \sum_{i=1}^\ell \alpha^*_{i:n}$.
    \item
    Assume now that $n-q \leq \ell \leq n-1$. Since the total sums of the coordinates of $\balpha$ and $\balpha^*$ are both equal to $c$, 
	$\sum_{i=1}^\ell \alpha_{i:n} = c - \sum_{i=\ell+1}^n \alpha_{i:n}$ and $\sum_{i=1}^\ell \alpha^*_{i:n} = c - \sum_{i=\ell+1}^n \alpha^*_{i:n}$.
	The number of terms in both complementary sums is $n-\ell \leq q$. Since each coordinate of $\balpha$ is bounded above by $b$, we have
    $\sum_{i=\ell+1}^n \alpha_{i:n} \le (n-\ell)b$.
	On the other hand, because the $n-\ell$ largest coordinates of $\balpha^*$ are all equal to $b$ in this range of $\ell$, we obtain
	$\sum_{i=\ell+1}^n \alpha^*_{i:n} = (n-\ell)b$.
	Hence
	$$
	\sum_{i=\ell+1}^n \alpha_{i:n} \leq \sum_{i=\ell+1}^n \alpha^*_{i:n}
	\quad\Rightarrow\quad
	\sum_{i=1}^\ell \alpha_{i:n} \geq \sum_{i=1}^\ell \alpha_{i:n}^*,
    $$
	which completes the proof.
    \end{enumerate}
\end{proof}

\subsection{Stochastic dominance between series systems}
Complex systems with components connected in series are treated similarly, after adapting appropriately the notation. The series system with components $X_i$, for $i=1,\ldots,n$, whose distributions are as described above, has lifetime $X_{1:n}(\boldsymbol{\alpha})=\min\{X_1,\ldots,X_n\}$, with tail $\Fbar_{1:n}^{(\balpha)}(x)={\rm P}(X_1>x,\ldots,X_n>x)$. As above, assuming that $\mathbf{X}$ is associated with an Archimedean copula with generator $\psi_1$ is given by
\begin{equation}
\label{eq:minX}
\widetilde{h}(\balpha;\psi)(x)=\Fbar_{1:n}^{(\balpha)}(x) = \psi_1\left(\sum_{i=1}^{n}\phi_1\left (1-T(\alpha_i,F(x))\right )\right)
= \psi_1\left(\sum_{i=1}^{n}\phi_1\left (\Tbar(\alpha_i,F(x))\right )\right),
\end{equation}
by defining $\overline{T}(\alpha,F)=1-T(\alpha,F)$. Analogous adaptations should be made to describe the lifetime $Y_{1:n}(\boldsymbol{\beta})$ of the series system with components $Y_i$, for $i=1,\ldots,n$.
We now obtain sufficient conditions for stochastic dominance between series system lifetimes by adapting the above results obtained in the previous subsection for parallel systems.

\medskip

\begin{theorem}
\label{thm:main0min}
{Assume the parameter vectors $\boldsymbol{\alpha}$ and $\boldsymbol{\beta}$ are such that $\alpha_i\leq\beta_i$, for $i=1,\ldots,n$}.
    \begin{enumerate}
    \item
    If $T$ is decreasing in its first argument that is, $\alpha \leq \beta$ implies that, for very $x\in\mathbb{R}$, $T(\alpha, F) (x) \geq T(\beta, F) (x)$ and the generators are such that $\phi_2\circ\psi_1$ is super-additive, then {$X_{1:n}(\boldsymbol{\alpha})\leqst Y_{1:n}(\boldsymbol{\beta})$}.
    \item
    If $T$ is increasing in its first argument that is, $\alpha\leq \beta$ implies that, for every $x\in\mathbb{R}$, $T(\alpha, F) (x) \leq T(\beta, F) (x)$ and the generators are such that $\phi_1\circ\psi_2$ is super-additive, then {$X_{1:n}(\boldsymbol{\alpha})\geqst Y_{1:n}(\boldsymbol{\beta})$}.
    \end{enumerate}
\end{theorem}
\begin{proof}
The proof follows the arguments used to prove Theorem~\ref{thm:main0} using the survival copula representation for the tails of the distributions (see (\ref{eq:cop-rep})).
\end{proof}

\begin{theorem}
\label{thm:main1min}
{Assume $T$ is differentiable with respect to its first argument, the generators are such that $\phi_2\circ\psi_1$ is super-additive, $\psi_1$ is $d$-monotone with $d\geq 3$, $\boldsymbol{\beta}\overset{m}\succeq\boldsymbol{\alpha}$} and 
\begin{equation}
\label{eq:V-inc-min}
\overline{V}(\beta) =\frac{1}{\psi_1^\prime\circ\phi_1(\Tbar(\beta,F))}\frac{\partial\Tbar}{\partial\beta}(\beta,F),\quad\beta\in A,
\end{equation}
is decreasing, then, {$X_{1:n}(\boldsymbol{\alpha})\leqst Y_{1:n}(\boldsymbol{\beta})$}.
\end{theorem}
\begin{proof}
According to Lemma~\ref{cop}, as $\phi_2\circ\psi_1$ is super-additive, it follows that {$\widetilde{h}(\boldsymbol{\beta};\psi_2) \geq \widetilde{h}(\boldsymbol{\beta};\psi_1)$}. 
Differentiating, we have
$$
\frac{\partial\widetilde{h}}{\partial\beta_i}=\psi_1^\prime\left(\sum_{i=1}^n\phi_1(\Tbar(\beta_i,F))\right)\overline{V}(\beta_i),
$$
hence
$$
(\beta_i-\beta_j)\left(\frac{\partial\widetilde{h}}{\partial\beta_{i}}-\frac{\partial\widetilde{h}}{\partial\beta_{j}}\right) =
\psi_1^\prime\left(\sum_{i=1}^n\phi_1(\Tbar(\beta_i,F))\right)(\beta_i-\beta_j)\bigl(\overline{V}(\beta_i)-\overline{V}(\beta_j)\bigr)\geq 0,
$$
as $\psi_1$ is decreasing, thus $\widetilde{h}$ is Schur-convex, so 
{$\widetilde{h}(\boldsymbol{\beta};\psi_1)\geq\widetilde{h}(\boldsymbol{\alpha};\psi_1)$}, concluding the proof.
%
\end{proof}

A statement similar to Corollary~\ref{cor:main1} is immediate.
\begin{corollary}
\label{cor:main1min}
Assume the same conditions as in Theorem~\ref{thm:main1min}. 
With $K_1(x)=-\frac{\psi_1(x)}{\psi_1^\prime(x)}$, the same conclusion holds if
(\ref{eq:V-inc-min}) is replaced by
\begin{equation}
\label{eq:K1-dec}
\overline{V}^\ast(\beta) =\frac{-K_1\circ\phi_1(\Tbar(\beta,F))}{\Tbar(\beta,F)}\frac{\partial\Tbar}{\partial\beta}(\beta,F),\quad\beta\in A,
\end{equation}
\end{corollary}

Corollary~\ref{cor:boxconst} admits a direct counterpart for series systems. Under identical parameter constraints, the extremal configurations with respect to stochastic ordering of the system lifetime can be characterized in a completely analogous way. The proof is omitted, as it follows the same line of reasoning as Corollary~\ref{cor:boxconst}.
\begin{corollary}
Assume the functional $T$ and the generator of the Archimedean copula describing the distribution of $\mathbf{X}$ are such that the function $V^\ast$ defined in (\ref{eq:K1-dec}) (with $\psi_2=\psi_1$) is decreasing. 
With the notations of Corollary~\ref{cor:boxconst}, for every $\balpha \in A_{a,b,c}$, we have that $X_{1:n}(\balpha^*) \leq_{st} X_{1:n}(\balpha) \leq_{st} X_{1:n}(\overline{\balpha})$.
\end{corollary}

\subsection{Stochastic dominance between $(n-k)$-out-of-$n$ systems}
We now consider the $(n-k)$-out-of-$n$ system with dependent and heterogeneous components, which remains functional whenever at least $n-k$ components are still alive. With the notation introduced at the beginning of Section~\ref{sec:general}, assuming the components lifetime distributions are $F_{X_i}=T(\alpha_i,F)$, the lifetime of such a system is given by $X_{n-k:n}(\balpha)$, where $\balpha=(\alpha_1,\ldots,\alpha_n)$. 
The derivation of a stochastic dominance result between $(n-k)$-out-of-$n$ systems, stated in Theorem~\ref{thm:heterogeneous-coordinatewise-comparison} below, requires a number of preparatory results, characterising the distribution of $X_{n-k:n}$ and some of its relevant properties, which we present first.
We start by an auxiliary result, helping with the representation of the relevant transformations.
\begin{lemma}\label{dif_gen}
Let $m \ge 1$, let $g:\mathbb{R}\longrightarrow\mathbb{R}$ be $m$ times differentiable, and $h_1,\dots,h_m>0$. Define the iterated forward difference operator by
\[
\Delta_{h_1,\dots,h_m} g(a) :=	(\Delta_{h_m}\circ\cdots\circ\Delta_{h_1})g(a), \qquad	\Delta_h g(a):=g(a+h)-g(a).
\]
Then the following identities hold:
\begin{align}
\Delta_{h_1,\dots,h_m} g(a) &=
   \sum_{L\subseteq\{1,\dots,m\}} (-1)^{m-|L|}\, g\left(a+\sum_{r\in L} h_r\right), \label{eq:FD-expansion} \\[1.5ex]
\Delta_{h_1,\dots,h_m} g(a) &=
   \int_0^{h_1}\cdots\int_0^{h_m} g^{(m)}(a+u_1+\cdots+u_m)\,du_1\cdots du_m. \label{eq:FD-integral}
\end{align}
\end{lemma}
\begin{proof}
We prove both identities by induction on $m$.
\begin{enumerate}[(i)]
\item
\textit{For \eqref{eq:FD-expansion}}. When $m=1$, we have $\Delta_h g(a)=g(a+h)-g(a)$, which coincides with \eqref{eq:FD-expansion}. Assume the result holds for $m-1$. Then
\[
\Delta_{h_1,\dots,h_m} g(a)
=
\Delta_{h_m}\bigl(\Delta_{h_1,\dots,h_{m-1}} g\bigr)(a)
=
\Delta_{h_1,\dots,h_{m-1}} g(a+h_m) - \Delta_{h_1,\dots,h_{m-1}} g(a).
\]
Applying the induction hypothesis to both terms, we obtain
\[
\sum_{L\subseteq\{1,\dots,m-1\}} (-1)^{(m-1)-|L|} g\left(a+h_m+\sum_{r\in L} h_r\right)
-
\sum_{L\subseteq\{1,\dots,m-1\}} (-1)^{(m-1)-|L|} g\left(a+\sum_{r\in L} h_r\right).
\]
Grouping terms according to whether the index $m$ belongs to the subset or not, we recover exactly the expansion \eqref{eq:FD-expansion} for $m$, with coefficients $(-1)^{m-|L|}$.

\item
\textit{For \eqref{eq:FD-integral}.} The representation is obvious for $m=1$. Assume now that \eqref{eq:FD-integral} is true for $m-1$.
Then
\[
\Delta_{h_1,\dots,h_m} g(a) =
\Delta_{h_m}
\left[ \int_0^{h_1}\cdots\int_0^{h_{m-1}} g^{(m-1)}(a+u_1+\cdots+u_{m-1}) \,du_1\cdots du_{m-1} \right].
\]
Applying $\Delta_{h_m}$ inside the integral and using the fundamental theorem of calculus, we obtain
\[
\int_0^{h_1}\cdots\int_0^{h_{m-1}} \int_0^{h_m}	g^{(m)}(a+u_1+\cdots+u_m) \,du_m\,du_1\cdots du_{m-1},
\]
concluding the proof.
\end{enumerate}
\end{proof}

We now proceed to the characterisation of the distribution function $F_{n-k:n}^{(\balpha)}$  of $X_{n-k:n}(\balpha)$. 
\begin{proposition}\label{prop:heterogeneous-n-k}
Let $1\le k\le n-1$ and let $X_1,\dots,X_n$ whose joint distribution is described by an Archimedean copula with generator $\psi$. For $x\in\mathbb{R}$ such that $0<T(\alpha_i,F)(x)<1$ for all $i$, define
\begin{equation}
\label{eq:u-and-t}
u_i(x):=T(\alpha_i,F)(x), \qquad t_i(x):=\phi\bigl(u_i(x)\bigr),\qquad i=1,\dots,n,
\end{equation}
Then, with $\balpha=(\alpha_1,\ldots,\alpha_n)$,
\begin{equation}
\label{cdf_gen}
F_{n-k:n}^{(\balpha)}(x) =
\sum_{j=n-k}^{n} (-1)^{j-(n-k)}\binom{j-1}{n-k-1} \sum_{\substack{I\subseteq\{1,\dots,n\}\\ |I|=j}}	\psi\left(\sum_{i\in I} t_i(x)\right).
\end{equation}
\end{proposition}
\begin{proof}
Set $B_i(x):=\{X_i\le x\}$, $i=1,\dots,n$. Then $\{X_{n-k:n}\le x\}=\left\{\sum_{i=1}^n \mathbf{1}_{B_i(x)}\ge n-k\right\}$.
By inclusion-exclusion for the probability that at least $n-k$ among the events $B_1(x),\dots,B_n(x)$ occur, we get
\[
\mathbb{P}\left(\sum_{i=1}^n \mathbf{1}_{B_i(x)}\ge n-k\right) =
\sum_{j=n-k}^{n} (-1)^{j-(n-k)}	\binom{j-1}{n-k-1} \sum_{\substack{I\subseteq\{1,\dots,n\}\\ |I|=j}} \mathbb{P}\left(\bigcap_{i\in I} B_i(x)\right).
\]
Since the joint distribution is described by the Archimedean copula $C_\psi$, if follows that
\[
\mathbb{P}\left(\bigcap_{i\in I} B_i(x)\right) =
C_\psi\bigl((u_i(x))_{i\in I}\bigr)=
\psi\left(\sum_{i\in I}\phi(u_i(x))\right)=	\psi\left(\sum_{i\in I} t_i(x)\right).
\]
Substituting this in the previous identity yields \eqref{cdf_gen}.
\end{proof}


\begin{proposition}\label{heterogeneous-derivative}
Assume the same conditions as in Proposition \ref{prop:heterogeneous-n-k}. Moreover, assume that $T$ is differentiable with respect to its first argument, that $\psi$ is $d$-monotone with $d\geq 3$, and that its inverse generator $\phi=\psi^{-1}$ is differentiable on the relevant interval. Then, with $\balpha=(\alpha_1,\ldots,\alpha_n)$, for each $\ell=1,\dots,n$,
\begin{equation}
\label{derivative_gen}	
\frac{\partial}{\partial \alpha_\ell}F_{n-k:n}^{(\balpha)}(x)=
\phi^\prime\bigl(u_\ell(x)\bigr)\,\frac{\partial T}{\partial \alpha}(\alpha_\ell,F)(x)
\sum_{j=n-k}^{n} (-1)^{j-(n-k)}	\binom{j-1}{n-k-1}	\sum_{\substack{I\subseteq\{1,\dots,n\}\\ |I|=j,\ \ell\in I}}	\psi^\prime\left(\sum_{i\in I} t_i(x)\right).
\end{equation}
\end{proposition}
\begin{proof}
First note that, when differentiating \eqref{cdf_gen} with respect to $\alpha_\ell$, only the subsets $I$ such that $\ell\in I$ contribute. For each such subset,
\[
\frac{\partial}{\partial \alpha_\ell}
\psi\left(\sum_{i\in I} t_i(x)\right) =
\psi^\prime\left(\sum_{i\in I} t_i(x)\right)\frac{\partial t_\ell(x)}{\partial \alpha_\ell}.
\]
Now, using (\ref{eq:u-and-t}) and differentiating concludes the proof.
\end{proof}

To be able to prove some convenient monotonicity results, that will imply the stochastic dominance relations between different $(n-k)$-out-of-$n$ systems, we need a different representation of the derivative (\ref{derivative_gen}). This is proved in the next result.
\begin{proposition}\label{heterogeneous-difference}
Assume the same conditions as in Proposition \ref{heterogeneous-derivative}. Fix $\ell\in\{1,\dots,n\}$ and write $N_\ell:=\{1,\dots,n\}\setminus\{\ell\}$.	Then
\begin{align}
&\sum_{j=n-k}^{n} (-1)^{j-(n-k)} \binom{j-1}{n-k-1}	\sum_{\substack{I\subseteq\{1,\dots,n\}\\ |I|=j,\ \ell\in I}}\psi^\prime\left(\sum_{i\in I} t_i(x)\right) \notag\\[1.5ex]
&\label{hetdif}\qquad\quad =	(-1)^k \sum_{\substack{J\subseteq N_\ell\\ |J|=n-k-1}} \Delta_{(t_i(x))_{i\in N_\ell\setminus J}}\psi^\prime\left(t_\ell(x)+\sum_{j\in J}t_j(x)\right),
\end{align}
where $\Delta_{(t_i(x))_{i\in N_\ell\setminus J}}$ denotes the iterated forward difference with increments $(t_i(x))_{i\in N_\ell\setminus J}$. Consequently,
\begin{align}
\label{hetder}
\frac{\partial}{\partial \alpha_\ell}F^{(\alpha)}_{n-k:n}(x) =
(-1)^k \phi^\prime(u_\ell(x))\,\frac{\partial T}{\partial \alpha}(\alpha_\ell,F)(x)
  \sum_{\substack{J\subseteq N_\ell\\ |J|=n-k-1}} \Delta_{(t_i(x))_{i\in N_\ell\setminus J}}
  \psi^\prime\left(t_\ell(x)+\sum_{j\in J}t_j(x)\right).
\end{align}
\end{proposition}
\begin{proof}
We only need to prove \eqref{hetdif}, as \eqref{hetder} is then an immediate from Proposition~\ref{heterogeneous-derivative}.
Fix $J\subseteq N_\ell$ with $|J|=n-k-1$. Then $|N_\ell\setminus J|=k$, and by	Lemma~\ref{dif_gen} and multiplying by $(-1)^k$, we find
\[
(-1)^k\Delta_{(t_i(x))_{i\in N_\ell\setminus J}}\psi^\prime\left(t_\ell(x)+\sum_{j\in J}t_j(x)\right)=
\sum_{L\subseteq N_\ell\setminus J}	(-1)^{|L|}	\psi^\prime\left(t_\ell(x)+\sum_{j\in J}t_j(x)+\sum_{i\in L}t_i(x)\right).
\]
We now determine the coefficient of a fixed term
\[
\psi^\prime\left(\sum_{i\in I} t_i(x)\right),\qquad \ell\in I,\quad |I|=m,
\]
in the right-hand side of \eqref{hetdif}. Denote $I=\{\ell\}\cup A$, $A\subseteq N_\ell$, $|A|=m-1$.
A term $\psi^\prime\left(\sum_{i\in I} t_i(x)\right)$ appears in the above expansion if and only if
\[
t_\ell(x)+\sum_{j\in J}t_j(x)+\sum_{i\in L}t_i(x)=t_\ell(x)+\sum_{a\in A}t_a(x),
\]
which is equivalent to $J\cup L = A$ and $J\cap L=\varnothing$. Hence such a term appears precisely when $J\subseteq A$ and $L=A\setminus J$. In addition, we must have $|J|=n-k-1$.
Thus the number of occurrences of the term is exactly the number of subsets $J\subseteq A$ with $|J|=n-k-1$, namely $\binom{m-1}{n-k-1}$.
For each such occurrence, the coefficient is $(-1)^{|L|}$. Since $|L|=|A|-|J|=(m-1)-(n-k-1)=m-(n-k)$, the sign of this coefficient is $(-1)^{|L|}=(-1)^{m-(n-k)}$.
Therefore, the total coefficient of $\psi^\prime\left(\sum_{i\in I} t_i(x)\right)$	in the right-hand side is
\[
(-1)^{m-(n-k)}\binom{m-1}{n-k-1},
\]
which coincides exactly with the coefficient of the same term in the left-hand side of \eqref{hetdif}.
Since this holds for every subset $I$ such that $\ell\in I$, the identity \eqref{hetdif} follows. This completes the proof.
\end{proof}

\begin{corollary}\label{heterogeneous-sign}
Assume the conditions of Proposition \ref{heterogeneous-difference}. Moreover, assume that the copula generator $\psi$ is $d$-monotone with $d\geq k+3$.
Then, for each $\ell=1,\dots,n$,
\begin{equation}
\label{sign}\operatorname{sgn}\!\left(
\frac{\partial}{\partial \alpha_\ell}F_{n-k:n}^{(\balpha)}(x)\right)=
\operatorname{sgn}\!\left(	\frac{\partial T}{\partial \alpha}(\alpha_\ell,F)(x)\right).
\end{equation}
\end{corollary}
\begin{proof}
We apply the integral representation in Lemma~\ref{dif_gen} to	each forward difference in \eqref{hetder}. Since $|N_\ell\setminus J|=k$, every such term has the form
\[
\int_0^{t_{i_1}(x)}\cdots\int_0^{t_{i_k}(x)}
\psi^{(k+1)}\left(t_\ell(x)+\sum_{j\in J}t_j(x)+u_1+\cdots+u_k\right)\,du_1\cdots du_k.
\]
By assumption, as $\psi$ is $d$-monotone, each of these integrals has sign $(-1)^{k+1}$. Therefore the whole sum over $J$ has sign $(-1)^{k+1}$, and multiplication by the factor $(-1)^k$ makes the bracket negative.
	Finally, since $\phi$ is decreasing, we have $\phi'(u_\ell(x))<0$. Hence
\[
\operatorname{sgn}\!\left(
\frac{\partial}{\partial \alpha_\ell}F_{n-k:n}^{(\balpha)}(x) \right)=
\operatorname{sgn}\!\left(\frac{\partial T}{\partial \alpha}(\alpha_\ell,F)(x)\right).
\]
\end{proof}

We may finally prove that the monotonicity of the functional $T$ implies stochastic dominance relations between the $(n-k)$-out-of-$n$ systems.
\begin{theorem}\label{thm:heterogeneous-coordinatewise-comparison}
Assume the conditions of Corollary \ref{heterogeneous-sign}.
    \begin{enumerate}
	\item 
    If $T$ is increasing in its first argument, that is, $\alpha\le \beta$ implies that $T(\alpha,F)(x)\le T(\beta,F)(x)$, for every $x\in\mathbb{R}$, then, for each fixed $x\in\mathbb{R}$, the map $\balpha=(\alpha_1,\dots,\alpha_n)\longmapsto F_{n-k:n}^{(\balpha)}(x)$	is increasing in each coordinate.
    
    \item 
    If $T$ is decreasing in its first argument, that is, $\alpha\le \beta$ implies that $T(\alpha,F)(x)\ge T(\beta,F)(x)$, for every $x\in\mathbb{R}$, then, for each fixed $x\in\mathbb{R}$, the same map is decreasing in each coordinate.
	\end{enumerate}
Consequently, if $\balpha=(\alpha_1,\ldots,\alpha_n)$ and $\bbeta=(\beta_1,\ldots,\beta_n)$ are such that $\alpha_i\le \beta_i$, for every $i=1,\dots,n$, then
    \begin{enumerate}
    \item
    If $T$ is increasing in its first argument, then $X_{n-k:n}(\balpha)\geqst X_{n-k:n}(\bbeta)$,
    \item
    If $T$ is decreasing in its first argument, then $X_{n-k:n}(\balpha)\leqst X_{n-k:n}(\bbeta)$.
	\end{enumerate}
\end{theorem}
\begin{proof}
We only prove assertion 1., as the decreasing case is handled analogously, reversing inequalities. Therefore, $T$ being increasing in its first argument means that $\frac{\partial T}{\partial \alpha}\ge 0$ whenever the derivative exists, so Corollary~\ref{heterogeneous-sign} yields $\frac{\partial}{\partial \alpha_\ell}F_{n-k:n}^{(\balpha)}(x)\ge 0$, for every $\ell=1,\dots,n$. Hence the map $\balpha\mapsto F_{n-k:n}^{(\balpha)}(x)$ is increasing in each coordinate.
	
Now suppose that $\alpha_i\le \beta_i$ for all $i=1,\ldots,n$, define the intermediate vectors $\boldsymbol{\gamma}^{(0)}:=(\alpha_1,\dots,\alpha_n)=\balpha$,
\[
\boldsymbol{\gamma}^{(m)}:=(\beta_1,\dots,\beta_m,\alpha_{m+1},\dots,\alpha_n),\qquad m=1,\dots,n-1,
\]
and $\boldsymbol{\gamma}^{(n)}=(\beta_1,\dots,\beta_n)=\bbeta$. By coordinatewise monotonicity, when $T$ is increasing we obtain
\[
F_{n-k:n}^{(\balpha)}(x)= F_{n-k:n}^{(\boldsymbol{\gamma}^{(0)})}(x)\le F_{n-k:n}^{(\boldsymbol{\gamma}^{(1)})}(x) \le \cdots \le	F_{n-k:n}^{(\boldsymbol{\gamma}^{(n)})}(x) = F_{n-k:n}^{(\bbeta)}(x),
\]
for every $x\in\mathbb{R}$, 
equivalently $X_{n-k:n}(\balpha)\geqst X_{n-k:n}(\bbeta)$.
\end{proof}
%

The following example shows that an extension of the previous result replacing the coordinatewise monotonicity of the parameters by a majorization relationship is not achievable.
\begin{example}
Take $\mathcal{E}(x)=1-e^{-x}$, for $x\geq 0$, the standard exponential, and the functional $T(\alpha,\mathcal{E})(x)=1-(1-\mathcal{E}(x))^\alpha=1-e^{-\alpha x}$, where $\alpha>0$. Consider a system with $n=3$ components with lifetimes $T(\alpha_i,\mathcal{E})$, for $i=1,2,3$, whose joint distribution is described by an Archimedean copula with Clayton generator $\psi(t)=(1+t)^{-1}$, for $t\geq 0$. With $\balpha=(\alpha_1,\alpha_2,\alpha_3)$, the lifetime of a 2-out-of-3 system is $X_{2:3}(\balpha)$, whose distribution function is obtained by using (\ref{cdf_gen}):
\[
F_{2:3}^{(\balpha)}(x)=\psi(t_1+t_2)+\psi(t_1+t_3)+\psi(t_2+t_3)-2\psi(t_1+t_2+t_3),
\]
where $t_i=(e^{\alpha_ix}-1)^{-1}$. A direct differentiation with respect to $\alpha_\ell$ gives
\[
\frac{\partial F^{(\balpha)}_{2:3}}{\partial \alpha_\ell}(x)
= t_\ell^\prime
\left(\sum_{j\ne \ell}\psi^\prime(t_\ell+t_j)- 2\psi^\prime(t_1+t_2+t_3)\right),
\]
where $t_\ell^\prime=-xe^{\alpha_\ell x}(e^{\alpha_\ell x}-1)^{-2}$. 
Choosing now $\balpha=(1,2.5,4)$, the plots for $(\alpha_\ell-\alpha_m)\left(\frac{\partial F^{(\balpha)}_{2:3}}{\partial \alpha_\ell}-\frac{\partial F^{(\balpha)}_{2:3}}{\partial \alpha_m}\right)$ are shown in Figure~\ref{fig:f1}.
\begin{figure}[h!]
\centering
\includegraphics[scale=.35]{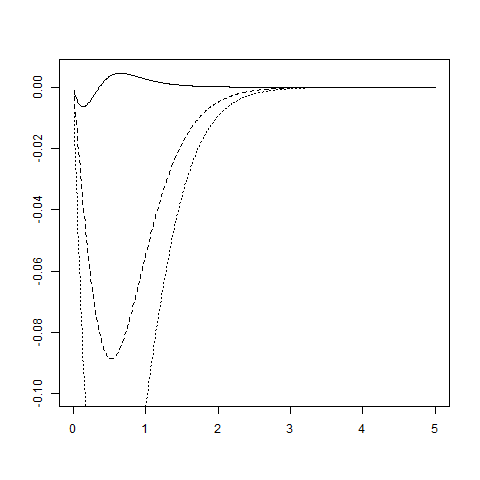}
\caption{Plots for $(\alpha_\ell-\alpha_m)\left(\frac{\partial F^{(\balpha)}_{2:3}}{\partial \alpha_\ell}-\frac{\partial F^{(\balpha)}_{2:3}}{\partial \alpha_m}\right)$: $(\ell,m)=(1,2)$ (solid), $(\ell,m)=(1,3)$ (dashed), $(\ell,m)=(2,3)$ (dotted).}
\label{fig:f1}
\end{figure}
Although when taking $(\ell,m)=(1,3)$ or $(2,3)$ the curve has a constant sign, the case $(\ell,m)=(1,2)$ changes sign as $x$ ranges $(0,+\infty)$. Hence, the map $\balpha\longmapsto F^{(\balpha)}_{2:3}(x)$ is neither Schur-convex nor Schur-concave.
\end{example} 

\section{Stochastic Comparisons for specific transformation models}

{We now demonstrate how the proposed framework can be applied to standard system structures. In many practical settings, systems consist of dependent and heterogeneous components whose distributions can be viewed as transformations of a common baseline model. In this context, the marginal distributions are expressed as $F_{X_i}=T(\alpha_i,F)$, where $F$ denotes the baseline distribution and the transformation $T$ reflects deviations across components. The parameters $\alpha_i$ therefore regulate the relative behaviour of individual components within a unified modelling scheme, while dependence among them is described through Archimedean copulas.}

{Within this setting, the maximum and minimum order statistics naturally represent system performance, corresponding to best- and worst-case scenarios, respectively. The results established in Section~\ref{sec:general} provide general conditions under which such systems can be compared in terms of stochastic ordering. In particular, monotonicity of the transformation $T$ leads to ordering results for parallel systems, whereas additional conditions involving the function $V$ allow majorization of the parameter vector to determine the ordering of extrema. In this way, the analysis makes explicit how heterogeneity in component characteristics influences overall system behaviour.}
From a methodological standpoint, the results highlight that stochastic ordering of system performance is driven by the interplay between three key elements: the monotonicity properties of the transformation, the dispersion of the parameter vector as described by majorization, and the dependence structure induced by the copula. This perspective provides a general mechanism that both extends and consolidates existing results, which are often confined to specific models or independence assumptions.

The properties of parallel systems will depend on the particular $V$ or $V^\ast$ function introduced in Theorem~\ref{thm:main1} and Corollary~\ref{cor:main1}, or their counterparts corresponding to series systems. Hence it is useful to have a description of the composition $\psi^\prime\circ\phi$, or of the quotient $-\tfrac{\psi}{\psi^\prime}$, where $\psi$ is the generator of each of the Archimedean copulas of interest and $\phi$ the corresponding inverse. These functions are given in Table~\ref{tab:Arch-fs}.

\begin{table}[h!]
	\centering
	\caption{Auxiliary functions for selected Archimedean copulas.}
	\label{tab:Arch-fs}
	\medskip
	\begin{tabular}{llll}
		\hline
		Copula & $\psi^\prime(t)$ & $\psi^\prime\circ\phi(t)$ & $K(t)=-\tfrac{\psi(t)}{\psi^\prime(t)}$ \\
		\hline
		Independence 
		& $e^{-t}$ 
		& $-t$ 
		& 1 \\[6pt]
		
		Clayton 
		& $-(1+\gamma t)^{-(\gamma+1)/\gamma}$ 
		& $-t^{\gamma+1}$ 
		& $1+\gamma t$ \\[10pt]

		Frank
		& $-\frac1\gamma\frac{e^{-t}(e^{-\gamma}-1)}{e^{-t}(e^{-\gamma}-1)+1}$ 
		& $\frac1\gamma(1-e^{\gamma t})$ 
		& $-\frac{e^{-t}(e^{-\gamma}-1)}{e^{-t}(e^{-\gamma}-1)+1} \log\left(e^{-t}(e^{-\gamma}-1)+1\right)$
        \\[10pt]
		
		Gumbel 
		& $-\frac1\gamma t^{(1-\gamma)/\gamma}\exp(-t^{1/\gamma})$ 
		& $-\frac{t}{\gamma}(-\log t)^{1-\gamma}$ 
		& $\gamma t^{(\gamma-1)/\gamma}$ \\[10pt]
		
		AMH 
		& $-\frac{(1-\gamma)e^t}{(e^t-\gamma)^2}$ 
		& $-\frac{t(1-\gamma+\gamma t)}{1-\gamma}$
		& $1-\gamma e^{-t}$ \\
		\hline
	\end{tabular}
\end{table}

\subsection{The PHR model}
As mentioned above, in Example~\ref{ex:exp} we considered a specific choice for the operator $T$ leading to exponential distributions with different hazard rates. {More generally, this construction gives rise to a parametric family of distributions that includes several widely used models}. Let us define, for $\beta>0$ and $F\in\PR$,
\begin{equation}
\label{eq:Te}
T_e(\beta,F)=1-(1-F)^\beta.
\end{equation}
This transformation corresponds to the PHR models with baseline distribution $F$, as the tail $1-T_e(\beta,F)=\Fbar^\beta$. It also includes the general K-G family introduced in \cite{Gauss2011}, which considers distribution functions of the form $1-(1-G^\alpha)^\beta$, where $G$ is some given distribution function, and $\alpha,\beta>0$. 
Thus, this framework includes the PHR model based on exponentiated distributions, while requiring weaker assumptions on the generators than those in \cite{PanjaKunduPradhan2021} and \cite{Datta_Gupta2024}. It is worth noting that the stochastic dominance result in \cite{Datta_Gupta2024} is derived through a combination of hazard rate and reversed hazard rate orderings.

We now apply the general results proved in Section~\ref{sec:general} to obtain stochastic dominance between different parallel or series systems with dependent heterogeneous components whose lifetimes are described through the operator $T_e$ and have joint distribution described by a convenient Archimedean copula. Note that, for each fixed $F\in\PR$, $\tfrac{\partial T_e}{\partial\beta}=-(1-F)^\beta\log(1-F)=-(1-T_e)\log(1-F)>0$, hence $T_e$ is increasing with respect to $\beta$.

\begin{proposition}
{With marginal distributions based on the functional $T_e$, assume the parameters $\balpha$ and $\bbeta$ satisfy $\alpha_i\leq\beta_i$, for $i=1,\ldots,n$, and the copula generators are such that $\phi_2\circ\psi_1$ is super-additive}. Then {$X_{n:n}(\balpha)\geqst Y_{n:n}(\bbeta)$}.
\end{proposition}
\begin{proof}
This is an immediate consequence of assertion 1.\ in Theorem~\ref{thm:main0}.
\end{proof}

If the parameter vectors satisfy the more relaxed majorization relation, we need to be more specific about the copulas describing the joint distributions.

\begin{proposition}
\label{prop:exp-parallel}
{With marginal distributions based on the functional $T_e$ and copula generators are such that $\phi_2\circ\psi_1$ is super-additive and $\psi_2$ is $d$-monotone with $d\geq 3$, assume the parameters satisfy $\balpha\overset{m}\succeq\bbeta$}. 
If any of the following conditions is fulfilled:
    \begin{enumerate}
    \item
    (Clayton copula) $\psi_2$ is the Clayton generator with $\gamma_2\geq -1$,
    \item
    (Gumbel copula) $\psi_2$ is the Gumbel generator with $\gamma_2\geq1$,
    \item
    (AMH copula) $\psi_2$ is the AMH generator with $\gamma_2\in(-1,1)$,
    \end{enumerate}
then {$X_{n:n}(\balpha)\geqst Y_{n:n}(\bbeta)$}.
\end{proposition}
\begin{proof}
Replacing the expression for the derivative of $T_e$ in (\ref{eq:V-inc}), means that we need to prove the increasingness with respect to $\beta>0$ of
$$
V(\beta)=\frac{-(1-T_e)\log(1-F)}{\psi_2^\prime\circ\phi_2(T_e)}.
$$
Given the monotonicity of $T_e$ with respect to $\beta$, the increasingness of $V(\beta)$ follows if we prove that $S_e(x)=\tfrac{1-x}{\psi_2^\prime\circ\phi_2(x)}$ is increasing for $x\in(0,1]$.
    \begin{enumerate}
    \item
    (Clayton copula) Replacing the expression for $\psi_2^\prime\circ\phi_2(x)$ from Table~\ref{tab:Arch-fs}, we find that $S_e(x)=x^{-\gamma_2}-x^{-(\gamma_2+1)}$, which is easily seen to be increasing whenever $\gamma_2\geq -1$ and $x\in(0,1]$.
    
    
    \item
    (Gumbel copula) We now have $S_e(x)=\gamma_2\left(1-\frac1x\right)\left(-\log x\right)^{\gamma_2-1}$. Differentiating shows that $S_e^\prime$ has the same sign as $-\log x-(\gamma_2-1)(x-1)$ that, for $\gamma_2\geq 1$ and $x\in(0,1]$, is nonnegative, thus $S_e$ is increasing.
    
    \item
    (AMH copula) In this case, we have $S_e(x)=\left(1-\frac1x\right)\tfrac{1-\gamma_2}{1-\gamma_2(1-x)}$. It easy verify that $S_e^\prime$ has the same sign as $1-\gamma_2+2\gamma_2 x-\gamma_2 x^2$. This quadratic has constant sign when $x\in(0,1]$ and $\gamma_2\in(-1,1)$, the range for the variation of the parameter, and the sign of this polynomial is positive. Hence, $S_e$ is indeed increasing when $x\in(0,1]$.
    \end{enumerate}
\end{proof}


We now describe conditions for stochastic dominance between series systems, {whose lifetime is given by $X_{1:n}(\balpha)=\min\{X_1,\ldots,X_n\}$ and $Y_{1:n}(\bbeta)=\min\{Y_1,\ldots,Y_n\}$, respectively}.
The following result is an immediate consequence of assertion 2.\ in Theorem~\ref{thm:main0min}.
\begin{proposition}
{With marginal distributions based on the functional $T_e$, copula generators such that $\phi_1\circ\psi_2$ is super-additive and parameters $\balpha$ and $\bbeta$ satisfying $\alpha_i\leq\beta_i$, for $i=1,\ldots,n$, we have that $X_{1:n}(\balpha)\geqst Y_{1:n}(\bbeta)$.}
\end{proposition}

\begin{proposition}
\label{prop:seriesPHR}
{With marginal distributions based on the functional $T_e$ and copula generators are such that $\phi_2\circ\psi_1$ is super-additive and $\psi_2$ is $d$-monotone with $d\geq 3$, assume the parameters satisfy $\bbeta\overset{m}\succeq\balpha$}. If any of the following conditions is fulfilled:
    \begin{enumerate}
    \item
    (Clayton copula) $\psi_1$ is the Clayton generator with $\gamma_1>0$,
    \item
    (Frank copula) $\psi_1$ is the Frank generator with $\gamma_1\in\mathbb{R}$,
    \item
    (Gumbel copula) $\psi_1$ is the Gumbel generator with $\gamma_1\geq1$,
    \item
    (AMH copula) $\psi_1$ is the AMH generator with $\gamma_1>0$,
    \end{enumerate}
then {$X_{1:n}(\balpha)\leqst Y_{1:n}(\bbeta)$}.
\end{proposition}
\begin{proof}
We have $\Tbar_e=(1-F)^\beta$, so $\tfrac{\partial \Tbar_e}{\partial\beta}=(1-F)^\beta\log(1-F)=\Tbar_e\log(1-F)<0$, implying that $\Tbar_e$ is decreasing with respect to $\beta$. Hence, according to Theorem~\ref{thm:main1min}, we need to prove that 
$$
\overline{V}(\beta)=\frac{\Tbar_e\log(1-F)}{\psi_1^\prime\circ\phi_1(\Tbar_e)}
$$
is decreasing. Now, this follows by proving that $\overline{S}_e=\tfrac{1-x}{\psi_1^\prime\circ\phi_1(1-x)}$ is decreasing for $x\in[0,1]$.
    \begin{enumerate}
    \item
    (Clayton copula) As $\overline{S}_e(x)=-(1-x)^{-\gamma_1}$, it is clear that $\overline{S}_e$ is decreasing when $\gamma_1>0$ and $x\in[0,1]$.
    
    \item
    (Frank copula) We now have $\overline{S}_e(x)=\tfrac{\gamma_1 (1-x)}{1-e^{\gamma_1(1-x)}}$, therefore $\overline{S}_e^\prime$ has the opposite sign as $e^{\gamma_1(1-x)}(1-\gamma_1(1-x))-1$ that can be seen to be nonnegative when $x\in[0,1]$, hence $\overline{S}_e$ decreasing.
    
    \item
    (Gumbel copula) For this copula we have $\overline{S}_e(x)=-\gamma_1(-\log(1-x))^{\gamma_1-1}$ so, 
    $$
    \overline{S}_e^\prime(x)=-\gamma_1(\gamma_1-1)(-\log(1-x))^{\gamma_1-2}\frac{1}{1-x}<0
    $$
    when $x\in[0,1]$ and $\gamma_1>1$. 
            
    \item
    (AMH copula) In this case $\overline{S}_e(x)=\tfrac{\gamma_1-1}{1-\gamma_1x}$, whose derivative has the same sign as $\gamma_1(\gamma_1-1)$, hence is negative for $\gamma_1>0$ given the range of variability of the parameter for the AMH copula.
    \end{enumerate}
\end{proof}

Stochastic dominance between $(n-k)$-out-of-$n$ systems with joint distributions based on the same Archimedean copula is an immediate consequence of Theorem~\ref{thm:heterogeneous-coordinatewise-comparison}. Indeed, according to latter result, all we need to verify is the monotonicity, with respect to the first parameter of $T_e$, that we already mentioned to be increasing. Hence the following statement holds.
\begin{proposition}
Let $1\leq k\leq n-1$, $\balpha=(\alpha_1,\ldots,\alpha_n)$ and $\bbeta=(\beta_1,\ldots,\beta_n)$, where $\alpha_i\leq\beta_i$, for every $i=1,\ldots,n$. With marginal distributions based on the functional $T_e$, if the copula generator $\psi_1$ is $d$-monotone with $d\geq k+3$ then $X_{n-k:n}(\balpha)\geqst X_{n-k:n}(\bbeta)$.
\end{proposition}

\subsection{The PRHR model}
A slight modification of the previous model considers
\begin{equation}
\label{eq:Te1}
T_c(\beta,F)=F^\beta,
\end{equation}
where $\beta>0$ and $F\in\PR$. This transformation defines the PRHR models that includes every exponentiated distribution function. {Note that we now have $\tfrac{\partial T_c}{\partial\beta}=T_c\log F<0$, hence $T_c$ is decreasing with respect to $\beta$.}
The following result is immediate taking into account assertion 2.\ in Theorem~\ref{thm:main0}.
\begin{proposition}
{With marginal distributions based on the functional $T_c$, copula generators such that $\phi_1\circ\psi_2$ is super-additive and parameters $\balpha$ and $\bbeta$ satisfying $\alpha_i\leq\beta_i$, for $i=1,\ldots,n$, we have that $X_{n:n}(\balpha)\leqst Y_{n:n}(\bbeta)$.}
\end{proposition}

\begin{proposition}
\label{prop:exp-parallel-c}
{With marginal distributions based on the functional $T_c$ and copula generators are such that $\phi_2\circ\psi_1$ is super-additive and $\psi_2$ is $d$-monotone with $d\geq 3$, assume the parameters satisfy $\balpha\overset{m}\succeq\bbeta$}. Assume any of the following conditions is fulfilled:
    \begin{enumerate}
    \item
    (Clayton copula) $\psi_2$ is the Clayton generator with $\gamma_2\geq 0$,
    \item
    (AMH copula) $\psi_2$ is the AMH generator with $\gamma_2\in[0,1]$.
    \end{enumerate}
Then, {$X_{n:n}(\balpha)\geqst Y_{n:n}(\bbeta)$}.
\end{proposition}
\begin{proof}
Replacing the derivative of $T_c$ in (\ref{eq:V-inc}), we need to prove the increasingness with respect to $\beta>0$ of
$$
V(\beta)=\frac{T_c\log F}{\psi_2^\prime\circ\phi_2(T_c)},
$$
which follows if we prove that $S_c(x)=\tfrac{x}{\psi_2^\prime\circ\phi_2(x)}$ is increasing for $x\in[0,1]$, taking into account that $\log F<0$.
    \begin{enumerate}
    \item
    (Clayton copula) We find that $S_c(x)=-x^{-\gamma_2}$, which is easily seen to be increasing whenever $\gamma_2\geq 0$ and $x\in[0,1]$.
    
%
    
    \item
    (AMH copula) In this case we have $S_c(x)=\tfrac{-(1-\gamma_2)}{1-\gamma_2+\gamma_2x}$ that is increasing when $\gamma_2\geq 0$ and $x\in[0,1]$. 
    \end{enumerate}
\end{proof}


As an immediate consequence of assertion 1.\ in Theorem~\ref{thm:main0min}, we have the following result.
\begin{proposition}
{With marginal distributions based on the functional $T_c$, copula generators such that $\phi_2\circ\psi_1$ is super-additive and parameters $\balpha$ and $\bbeta$ satisfying $\alpha_i\leq\beta_i$, for $i=1,\ldots,n$, we have that $X_{1:n}(\balpha)\leqst Y_{1:n}(\bbeta)$.}
\end{proposition}

For the copulas considered, the series systems based on distributions issued from the PRHR model are not comparable with respect to stochastic domination when the parameter vectors only satisfy a majorization relation.

As for the PHR model we may prove the stochastic dominance between $(n-k)$-out-of-$n$ systems when the joint distribution is described by the same Archimedean copula, as an immediate consequence of Theorem~\ref{thm:heterogeneous-coordinatewise-comparison}. Taking into account that $T_c$ is decreasing with respect to its first argument, the following statement is immediate.
\begin{proposition}
Let $1\leq k\leq n-1$, $\balpha=(\alpha_1,\ldots,\alpha_n)$ and $\bbeta=(\beta_1,\ldots,\beta_n)$, where $\alpha_i\leq\beta_i$, for every $i=1,\ldots,n$. With marginal distributions based on the functional $T_c$, if the copula generator $\psi_1$ is $d$-monotone with $d\geq k+3$ then $X_{n-k:n}(\balpha)\leqst X_{n-k:n}(\bbeta)$.
\end{proposition}

\subsection{The odds-Marshall-Olkin model}
The odds-Marshall-Olkin (oMO) model, introduced in \cite{AHO2025}, defines a flexible class of distributions obtained by combining the proportionality of the odds function with the proportionality of the log-odds function, defining a parametric class of distributions constructed from a baseline distribution. Recall that the odds function of the distribution function $F$ is $\Lambda_F(x)=\frac{F(x)}{\Fbar(x)}$. The distributions in the oMO class, denoted in \cite{AHO2025} as $G_{\beta,\theta}$, where $\beta,\,\theta>0$ are defined by $\Lambda_{\nsG}=\beta\Lambda_F^\theta$, which allows for the explicit representation
\begin{equation}
\label{eq:T}
\nsG(x) = \ToMO(\beta,F)=\dfrac{\beta F^\theta(x)}{\beta F^\theta(x) + \Fbar^\theta(x)}.
\end{equation}
This construction extends the classical Marshall-Olkin proportional odds model ($\theta=1$) and also includes proportional log-odds models ($\beta=1$). 
The two parameters allow separate control over scale and curvature of the odds function, thereby offering flexibility in skewness, tail behaviour and growth rate of the hazard.

In what follows, we derive conditions for stochastic dominance between parallel or series systems with heterogeneous components whose lifetimes follow some oMO distribution. Note that, as the odds function ranges from 0 to $+\infty$, distributions within the oMO class are comparable only when they share the same value for the parameter $\theta$. For this reason, we shall concentrate on the variation with respect to the parameter $\beta$. 

We will characterise the stochastic dominance between parallel systems when the joint distribution of the components lifetimes is described by each of the copula families, identifying the range of parameters for which it holds. Note that $\tfrac{\partial \ToMO}{\partial\beta}=\tfrac{\ToMO^2}{\beta^2}\left(\tfrac{\Fbar}{F}\right)^\theta$, so $\ToMO$ is increasing with respect to $\beta>0$. Therefore, the following dominance result is a straightforward consequence of Theorem~\ref{thm:main0}.
\begin{proposition}
{With marginal distributions based on the functional $\ToMO$, copula generators such that $\phi_2\circ\psi_1$ is super-additive and parameters $\balpha$ and $\bbeta$ satisfying $\alpha_i\leq\beta_i$, for $i=1,\ldots,n$, we have that $X_{n:n}(\balpha)\geqst Y_{n:n}(\bbeta)$.}
\end{proposition}

\begin{proposition}
\label{prop:parallel}
{With marginal distributions based on the functional $\ToMO$ and copula generators are such that $\phi_2\circ\psi_1$ is super-additive and $\psi_2$ is $d$-monotone with $d\geq 3$, and parameters satisfying $\balpha\overset{m}\succeq\bbeta$,} assume any of the following conditions is fulfilled:
    \begin{enumerate}
    \item
    (Clayton copula) $\psi_2$ is the Clayton generator with $\gamma_2\geq1$,
    \item
    (Frank copula) $\psi_2$ is the Frank generator with $\gamma_2\geq 0$,
    \item
    (Gumbel copula) $\psi_2$ is the Gumbel generator with $\gamma_2\geq1$,
    \item
    (AMH copula) $\psi_2$ is the AMH generator with $\gamma_2\in(-1,1)$,
    \end{enumerate}
then {$X_{n:n}(\balpha)\geqst Y_{n:n}(\bbeta)$}.
\end{proposition}
\begin{proof}
For each considered copula we will apply Corollary~\ref{cor:main1}, verifying that in every case $V^\ast$ is increasing with respect to $\beta>0$. 
    \begin{enumerate}
    \item
    (Clayton copula) In this case we have $K_2(x)=-\tfrac{\psi_2(x)}{\psi_2^\prime(x)}=1+\gamma_2 x$, hence, $V^\ast(\beta)=-\tfrac{1}{\beta^2}\ToMO^{1-\gamma_2}(\beta,F)=-\tfrac{1}{\beta^2}\nsG^{1-\gamma_2}$. A simple differentiation shows that $V^\ast$ is increasing with respect to $\beta$, concluding the proof.
    
    \item
    (Frank copula) We now have
    $$
    K_2(x)=\frac{1+e^{-x}(e^{-\gamma_2}-1)}{e^{-x}(e^{-\gamma_2}-1)}\log\left(1+e^{-x}(e^{-\gamma_2}-1)\right),
    $$
    hence, using (\ref{eq:K2-inc}) and replacing (\ref{eq:T}), we find that
    $$
    V^\ast(\beta)=-\frac{1}{\left(1-e^{-\gamma_2 \ToMO(\beta,F)}\right)\left(\beta F^\theta+\Fbar^\theta\right)^2}\frac{F^{3\theta}}{\Fbar^\theta}.
    $$
    As we are assuming that $\gamma_2>0$ and $\ToMO$ is increasing with respect to $\beta>0$, the denominator is clearly increasing, so the same monotonicity holds for $V^\ast$. 
    
    \item
    (Gumbel copula) For this copula choice, we have $K_2(x)=-\tfrac{\psi_2(x)}{\psi_2^\prime(x)}=\tfrac{x^{1-a}}{a}$, where $a=\tfrac1{\gamma_2}$, so 
    $$
    V^\ast(\beta)=-\frac{\ToMO(\beta,F)}{\beta^2}\Bigl(-\log \ToMO(\beta,F)\Bigr)^{\frac{1-a}a}\frac{\Fbar^\theta}{aF^\theta}.
    $$
    Set $Q(\beta) = \tfrac{u(-\log u)^{\frac{1-a}{a}}}{\beta^2}=-V^\ast(\beta)$, where $u = \ToMO(\beta,F)=\nsG$. Noting that $\log Q(\beta) = \log u+\tfrac{1-a}{a}\log(-\log u) - 2\log \beta$, it follows, with $k = \frac{1-a}{a}=\gamma_2-1\geq0$,
    $$
    \dfrac{d\log Q(\beta)}{d \beta} = \dfrac{u^\prime}{u}+ \frac{1-a}{a}\left (\dfrac{u^\prime}{u}\right )\left (\dfrac{1}{\log u}\right )-\dfrac{2}{\beta} =  \dfrac{u^\prime}{u}\left (1+\dfrac{k}{\log u}\right ) - \dfrac{2}{\beta}
    \leq \dfrac{u^\prime}{u} -\dfrac{2}{\beta}.
    $$
    Notice that 
    \[
    0\leq \dfrac{u^\prime}{u} = \dfrac{1}{\nsG}\dfrac{d\nsG}{d \beta} = \dfrac{\Fbar^\theta}{\beta\left(\beta F^\theta + \Fbar^\theta\right)}\leq \dfrac{1}{\beta}
    \]
    and therefore,
    \[
    \dfrac{Q^\prime(\beta)}{Q(\beta)} \leq \dfrac{u^\prime}{u} -\dfrac{2}{\beta} \leq -\dfrac{1}{\beta}<0.
    \]
    Since $Q(\beta)>0$ we conclude that $Q(\beta)$ is decreasing in $\beta$, hence $V^\ast(\beta)=-Q(\beta)$ is increasing. 
    
    \item
    (AMH copula) We have $K_2(x)=1-\gamma_2 e^{-x}$ so, after simplification, and replacing  with (\ref{eq:T}),
    $$
    V^\ast(\beta)=\frac{-(1-\gamma_2)}{1-\gamma_2(1-\ToMO(\beta,F))}\tfrac{\ToMO(\beta,F)}{\beta^2}
    =\frac{-(1-\gamma_2)F^\theta}{\beta\left(\beta F^\theta+(1-\gamma_2)\Fbar^\theta\right)},
    $$
    which is clearly increasing with respect to $\beta>0$. 
    \end{enumerate}
\end{proof}

\begin{example}
We showed in Example~\ref{ex:fams} that the super-additivity assumption is fulfilled composing two different copulas within the same family. We now show that the Gumbel family allows for some more flexibility on the choice of the copula $\psi_1$. Choose $\psi_1(t) = \exp\!\left(-t - c t^2\right)$, for some $c > 0$. Clearly, $\psi_1$ is decreasing, $\psi_1(0)=1$ and $\psi_1(t)\to 0$ as $t\to\infty$. Moreover, it is completely monotone and therefore defines a valid Archimedean copula. Then,
\[
\phi_2 \circ \psi_1(t)= \bigl(-\log \psi_1(t)\bigr)^{\gamma_2} = (t + c t^2)^{\gamma_2}.
\]
Taking $\gamma_2>1$, this is an increasing and convex transformation of a super-additive function, hence it is super-additive.
\end{example}


\begin{example}
Concerning the super-additivity assumption in Proposition~\ref{prop:parallel} when using a AMH generator, we have already verified that we may choose for $\psi_1$ the independent copula or another copula within the AMH family. We show next that we may also take $\psi_1$ within the Gumbel class.
Let $\psi_1(t)=\exp\big(-t^{1/\alpha}\big)$, $\alpha\ge 1$, be the generator function of the Gumbel family of copulas, and assume $\gamma_2>0$. Then
$$
\phi_2\circ\psi_1(t)
	=\log\left(\frac{1-\gamma_2(1-e^{-t^{1/\alpha}})}{e^{-t^{1/\alpha}}}\right)
	=\log\left((1-\gamma_2)e^{t^{1/\alpha}}+\gamma_2\right)
	=:g_\alpha(t).
$$
For $x,y\ge 0$ we again reduce super-additivity to an inequality without logarithms:
$$
g_\alpha(x+y)\ge g_\alpha(x)+g_\alpha(y)
\quad\Leftrightarrow\quad
(1-\gamma_2)e^{(x+y)^{1/\alpha}}+\gamma_2
\ge
\big((1-\gamma_2)e^{x^{1/\alpha}}+\gamma_2\big)\big((1-\gamma_2)e^{y^{1/\alpha}}+\gamma_2\big).
$$
Expanding the product and rearranging, it suffices to show
$$
(1-\gamma_2)e^{(x+y)^{1/\alpha}}-(1-\gamma_2)^2 e^{x^{1/\alpha}+y^{1/\alpha}}
\ge
\gamma_2(1-\gamma_2)\big(e^{x^{1/\alpha}}+e^{y^{1/\alpha}}-1\big).
$$
Since $\alpha\ge 1$, the map $t\mapsto t^{1/\alpha}$ is subadditive, hence
\[
(x+y)^{1/\alpha}\le x^{1/\alpha}+y^{1/\alpha}
\quad\Rightarrow\quad
e^{(x+y)^{1/\alpha}}\le e^{x^{1/\alpha}+y^{1/\alpha}}.
\]
Therefore,
\begin{eqnarray*}
	\lefteqn{(1-\gamma_2)e^{(x+y)^{1/\alpha}}+\gamma_2
		-\big((1-\gamma_2)e^{x^{1/\alpha}}+\gamma_2\big)\big((1-\gamma_2)e^{y^{1/\alpha}}+\gamma_2\big)}\\[1.5ex]
	&\ge &
		\gamma_2(1-\gamma_2)\Big(e^{x^{1/\alpha}+y^{1/\alpha}}-e^{x^{1/\alpha}}-e^{y^{1/\alpha}}+1\Big)
		=\gamma_2(1-\gamma_2)\big(e^{x^{1/\alpha}}-1\big)\big(e^{y^{1/\alpha}}-1\big)\ge 0,
\end{eqnarray*}
which proves that $g_\alpha$ is super-additive.
\end{example}


We now derive conditions and the range of variability of the copula parameters that allow for stochastic domination between the lifetime of series systems with dependent components. Recall that we obtained the representation for the tail of the distribution of the system lifetime $X_{1:n}(\balpha)=\min\{X_1,\ldots,X_n\}$ given in (\ref{eq:minX}).
Given the increasingness of $\ToMO$ with respect to $\beta>0$, we immediately have the following result.
\begin{proposition}
{With marginal distributions based on the functional $\ToMO$, copula generators such that $\phi_2\circ\psi_1$ is super-additive and parameters $\balpha$ and $\bbeta$ satisfying $\alpha_i\leq\beta_i$, for $i=1,\ldots,n$, we have that $X_{1:n}(\balpha)\geqst Y_{1:n}(\bbeta)$.}
\end{proposition}

\begin{proposition}
\label{prop:seriesoMO}
{With marginal distributions based on the functional $\ToMO$, $\psi_1$ the Clayton generator with $\gamma_1\leq 1$, $\psi_2$ a copula generator such that $\phi_2\circ\psi_1$ is super-additive and $\psi_2$ is $d$-monotone with $d\geq 3$, and parameters satisfying $\bbeta\overset{m}\succeq\balpha$,} 
then {$X_{1:n}(\balpha)\leqst Y_{1:n}(\bbeta)$}.
\end{proposition}
\begin{proof}
Similarly to the previous result, according to Theorem~\ref{thm:main1min}, we shall verify that in each case $\overline{V}$ is decreasing with respect to $\beta>0$. We now have $\tfrac{\partial\Tbar_o}{\partial\beta}=-\Tbar_o^2\tfrac{F^\theta}{\Fbar^\theta}$, hence, $\Tbar_o$ is decreasing with $\beta>0$. Therefore,
$$
\overline{V}(\beta)=\frac{-\Tbar_o^2}{\psi_1^\prime\circ\phi_1(\Tbar_o)}\left(\frac{F}{\Fbar}\right)^\theta,
$$
which is decreasing if $\overline{S}_{o}=\frac{x^2}{\psi_1^\prime\circ\phi_1(x)}$ is proven to be decreasing for $x\in[0,1]$. 
For the Clayton generator, according to Table~\ref{tab:Arch-fs}, we have $\overline{S}_o(x)=-x^{1-\gamma_1}$, which is decreasing when $x\in[0,1]$ and $\gamma_1\leq1$.
%
%
%
\end{proof}

The stochastic dominance between $(n-k)$-out-of-$n$ systems with joint distribution described by the same Archimedean copula is again an immediate consequence of Theorem~\ref{thm:heterogeneous-coordinatewise-comparison}. Taking into account that $\ToMO$ is increasing with respect to its first argument, the following statement is immediate.
\begin{proposition}
Let $1\leq k\leq n-1$, $\balpha=(\alpha_1,\ldots,\alpha_n)$ and $\bbeta=(\beta_1,\ldots,\beta_n)$, where $\alpha_i\leq\beta_i$, for every $i=1,\ldots,n$. With marginal distributions based on the functional $\ToMO$, if the copula generator $\psi_1$ is $d$-monotone with $d\geq k+3$ then $X_{n-k:n}(\balpha)\geqst X_{n-k:n}(\bbeta)$.
\end{proposition}

\bibliographystyle{abbrvnat-nourl} 
\bibliography{biblio}  

%
%
%
%
%
\end{document}